\documentclass[a4paper,11pt]{amsart}

\usepackage[T1]{fontenc}
\usepackage{lmodern}
\usepackage[utf8]{inputenc}
\usepackage[english]{babel}
\usepackage{amsmath,amssymb,amsfonts,amsthm,mathtools}
\usepackage{enumitem}
\usepackage{microtype}
\usepackage{geometry}
\usepackage[colorlinks=true,linkcolor=blue,citecolor=blue,urlcolor=blue]{hyperref}

\usepackage{cite}

\numberwithin{equation}{section}

\newtheorem{theorem}{Theorem}[section]
\newtheorem{proposition}[theorem]{Proposition}
\newtheorem{lemma}[theorem]{Lemma}
\newtheorem{corollary}[theorem]{Corollary}
\newtheorem{remark}[theorem]{Remark}

\newcommand{\R}{\mathbb{R}}
\newcommand{\Z}{\mathbb{Z}}
\newcommand{\eps}{\varepsilon}
\newcommand{\cM}{\mathcal{M}}

\newcommand{\weak}{\rightharpoonup}
\newcommand{\norm}[1]{\left\lVert #1\right\rVert}

\title[Large defocusing parameter]{Existence of solutions for a large defocusing parameter in a strongly indefinite Schr\"odinger equation}

\author[B. Bieganowski]{Bartosz Bieganowski}

\address[B. Bieganowski]{\newline\indent
	Faculty of Mathematics, Informatics and Mechanics, \newline\indent
	University of Warsaw, \newline\indent
	ul. Banacha 2, 02-097 Warsaw, Poland}
\email{\href{mailto:bartoszb@mimuw.edu.pl}{bartoszb@mimuw.edu.pl}}

\date{}

\begin{document}

\begin{abstract}
Let $N\geq 3$, $2<q<p<2^*$, and let $V\in L^\infty(\R^N)$ be real-valued and $\Z^N$-periodic. We assume that $0$ belongs to a finite spectral gap of the periodic Schr\"odinger operator $L=-\Delta+V$. We prove that there exists $\lambda_\infty>0$ such that, for every $\lambda\geq\lambda_\infty$, the competing-power equation
$$
-\Delta u+V(x)u=|u|^{p-2}u-\lambda |u|^{q-2}u
\quad\text{in }\R^N
$$
has a nontrivial solution $u_\lambda\in H^1(\R^N)$. Moreover,
$$
\norm{u_\lambda}_{H^1(\R^N)}+\norm{u_\lambda}_{L^\infty(\R^N)}
\leq C\lambda^{-1/(q-2)}.
$$
We also show that, after the natural rescaling and lattice translations, a sequence of such solutions converges weakly to a nontrivial solution of the pure defocusing equation $Lv=-|v|^{q-2}v$.

\medskip

   

\end{abstract}

\maketitle

\section{Introduction and main result}

We consider the following stationary Schr\"odinger equation
\begin{equation}\label{eq:original}
Lu:=(-\Delta+V(x))u
=|u|^{p-2}u-\lambda |u|^{q-2}u,
\qquad x\in\R^N,
\end{equation}
where
\begin{equation}\label{eq:pq}
N\geq3,
\qquad
2<q<p<2^*:=\frac{2N}{N-2}.
\end{equation}

Equation \eqref{eq:original} describes the so-called \textit{standing waves}, namely solutions of the form
$$
\psi(t,x) = e^{\mathrm{i}\omega t} u(x)
$$
to the time-dependent Schr\"odinger equation
$$
\mathrm{i} \partial_t \psi = -\Delta \psi + (V(x) - \omega) \psi - h(|\psi|) \psi,
$$
and appears in nonlinear optics. When $V$ is a periodic potential, it finds in applications in the theory of photonic crystals. Then $V$ describes the periodic nano-structure of the material and the nonlinearity is responsible for the nonlinear polarization of the medium. For the details we refer the reader to \cite{Pankov2005, Slusher, Kuchment, Buryak}. In particular, when we consider the term of the form $|u|^{p-2}u - \lambda |u|^{q-2}q$ we deal with a mixture of self-focusing and defocusing materials.

Motivated by this application, throughout the paper we assume that
\begin{equation}\label{eq:V}
V\in L^\infty(\R^N),\qquad
V(x+k)=V(x)\quad\text{for a.e. }x\in\R^N,\ k\in\Z^N,
\end{equation}
and that $0$ lies in a finite gap of the spectrum of the Schr\"odinger operator
$$
L=-\Delta+V
\quad\text{in }L^2(\R^N).
$$
More precisely, 
\begin{equation}\label{eq:gap}
0 \not\in \sigma(L),
\qquad
\sigma(L)\cap(-\infty,0)\ne\emptyset,
\end{equation}
where $\sigma(L)$ denotes the spectrum of $L$. In such a setting, solutions $u$ are also known as \textit{gap solitons}. 
Thus the problem is genuinely strongly indefinite. Note that $V$ cannot be constant. Indeed, if $V \equiv V_0 \in \R$, then $\sigma(L) = [V_0,\infty)$ and there are no spectral gaps. Variational methods for nonlinear elliptic equations and Hamiltonian systems have their classical roots in the mountain-pass and minimax theory of Ambrosetti-Rabinowitz and Rabinowitz \cite{AmbrosettiRabinowitz1973,Rabinowitz1978,Rabinowitz1986} and in the Nehari constraint \cite{Nehari1960}, see also \cite{CotiZelatiRabinowitz1991, Jeanjean1999, JeanjeanTanaka2003}. In the strongly indefinite setting, where the quadratic part has infinite-dimensional positive and negative spectral subspaces, the usual finite-dimensional linking picture has to be replaced by genuinely infinite-dimensional constructions. An influential abstract result in this direction is the generalized linking theorem of Kryszewski and Szulkin \cite{KryszewskiSzulkin1998}.

For periodic nonlinear Schr\"odinger equations, strongly indefinite variational methods were developed in a series of works. Bartsch and Ding treated periodic Schr\"odinger equations by deformation and linking methods \cite{BartschDing1999}; Pankov developed the generalized Nehari manifold approach \cite{Pankov2005}. Refinements of these methods can be found, among others, in \cite{LiWangZeng2006,SzulkinWeth2009,Liu2012,Mederski2016,dePaivaKryszewskiSzulkin2017}. In particular, the reduction of Szulkin and Weth \cite{SzulkinWeth2009} gives a convenient minimax characterization of the ground-state level in the strongly indefinite case, while de Paiva, Kryszewski and Szulkin \cite{dePaivaKryszewskiSzulkin2017} showed how the method can be adapted under a weak monotonicity assumption. These developments form the variational background for the construction used below.

A separate difficulty appears when the nonlinear part itself changes sign. Chen and Wang \cite{ChenWang2014} obtained an infinite-dimensional linking theorem without the upper semicontinuity assumption occurring in the classical Kryszewski-Szulkin framework and applied it to a periodic strongly indefinite Schr\"odinger equation with sign-changing nonlinearity. For Schr\"odinger equations with a positive linear part, sign-changing nonlinearities and periodic potentials were studied by Bieganowski and Mederski \cite{BieganowskiMederski2018}, who obtained ground states by a Nehari-type minimization. In the genuinely strongly indefinite case, Bernini and Bieganowski \cite{BerniniBieganowski2022} developed a generalized linking theorem designed specifically for functionals with a sign-changing nonlinear part; this framework was subsequently complemented by an abstract multiplicity theory for critical orbits in dislocation spaces by Bernini, Bieganowski and Strzelecki \cite{BerniniBieganowskiStrzelecki2025}. 

The model considered here is the particularly transparent competing-power equation \eqref{eq:original}. Its nonlinear primitive
$$
\frac{|t|^p}{p}-\lambda\frac{|t|^q}{q}
$$
is negative near the origin and positive for large $|t|$. For sufficiently small $\lambda$, existence is available from generalized linking methods for sign-changing nonlinearities; see in particular \cite{BerniniBieganowski2022} and the discussion in \cite{Bieganowski2026}. A least-energy result for the pure-power problem was recently obtained in \cite{Bieganowski2026}. 

The purpose of the present paper is to analyze precisely the opposite regime $\lambda\to\infty$. We prove that solutions do not disappear for large defocusing part of the nonlinearity, namely for every sufficiently large $\lambda$ there exists a nontrivial solution, and its $H^1$- and $L^\infty$-norms are $O(\lambda^{-1/(q-2)})$. After the scaling associated with the $q$-term, the equation approaches the pure defocusing case
$$
Lv=-|v|^{q-2}v,
$$
which is covered \cite{Pankov2005}. The main point is that we do not perturb a fixed limiting solution and therefore require no nondegeneracy. Instead we truncate the higher power, interchange the positive and negative spectral spaces, apply a generalized Nehari-Pankov method, and obtain estimates uniform in the truncation parameter. Concentration compactness in the spirit of Lions \cite{Lions1984} then handles the loss of compactness caused by lattice translations. The truncation is eventually inactive, which yields a solution of the original competing equation.

The main theorem is the following.

\begin{theorem}\label{thm:main}
Assume \eqref{eq:pq}, \eqref{eq:V}, and \eqref{eq:gap}. Then there exist constants
$$
\lambda_\infty>0,
\qquad C>0,
$$
depending only on $N,p,q,V$, such that for every $\lambda\geq\lambda_\infty$ equation \eqref{eq:original} possesses a nontrivial weak solution
$$
u_\lambda\in H^1(\R^N)\cap L^\infty(\R^N).
$$
Moreover,
$$
\norm{u_\lambda}_{H^1(\R^N)}
+\norm{u_\lambda}_{L^\infty(\R^N)}
\leq C\lambda^{-1/(q-2)}.
$$
\end{theorem}

The proof has five steps.
\begin{enumerate}[label=\textup{(\roman*)}]
\item We scale the equation so that $\lambda\to\infty$ becomes a small parameter $\eps\to0$ multiplying the $p$-term.
\item We truncate the $p$-term at an $\eps$-dependent amplitude $M_\eps\to\infty$. The resulting nonlinearity has one sign and satisfies uniform $q$-power bounds.
\item The defocusing truncated equation is rewritten as a standard strongly indefinite focusing functional after interchanging the two spectral spaces. The generalized Nehari-Pankov reduction gives a ground state $v_\eps$.
\item A comparison of the generalized Nehari minimax levels with the pure $q$-power problem yields bounds for $v_\eps$ in $H^1$ that are uniform in $\eps$. A subcritical elliptic bootstrap then yields a uniform $L^\infty$ bound.
\item Since $M_\eps\to\infty$, the uniform $L^\infty$ bound implies that the truncation is inactive for sufficiently small $\eps$. Thus $v_\eps$ solves the untruncated scaled equation, and scaling back gives \eqref{eq:original}.
\end{enumerate}

We also study what happens when $\lambda \to \infty$.

\begin{theorem}\label{thm:profile}
Let $\lambda_n\to\infty$, set
$$
\eps_n=\lambda_n^{-(p-2)/(q-2)},
\qquad
v_n=\lambda_n^{1/(q-2)}u_{\lambda_n},
$$
where $u_{\lambda_n} \in H^1 (\R^N)$ is the solution found in Theorem \ref{thm:main}. After passing to a subsequence, there exist $k_n\in\Z^N$ and a nontrivial $v_0\in H^1(\R^N)$ such that
\begin{align}
v_n(\cdot+k_n)&\weak v_0
&&\text{in }H^1(\R^N),\label{eq:profile-weak}\\
v_n(\cdot+k_n)&\to v_0
&&\text{in }L^r_{\mathrm{loc}}(\R^N),\quad 2\leq r<2^*,\label{eq:profile-local}
\end{align}
and $v_0$ weakly solves
\begin{equation}\label{eq:profile-eq}
Lv_0=-|v_0|^{q-2}v_0.
\end{equation}
\end{theorem}

In \cite{BerniniBieganowski2022}, it was shown that there exists $\lambda_{*}>0$ such that a nontrivial solution exists for every $\lambda\leq \lambda_{*}$, whereas Theorem \ref{thm:main} establishes existence for all sufficiently large parameters, namely for $\lambda\geq\lambda_{\infty}$. This leaves open the natural and intriguing question of what happens in the intermediate regime $\lambda \in ( \lambda_{*}, \lambda_{\infty})$.

\section{Spectral splitting and the large-\texorpdfstring{$\lambda$}{lambda} scaling}
\label{sect:2}

Let $P^+$ and $P^-$ denote the spectral projections on $L^2(\R^N)$ corresponding to $\sigma(L) \cap (0,\infty)$ and $\sigma(L) \cap (-\infty,0)$, respectively. By \eqref{eq:gap},
$$
H^1(\R^N)=X^+\oplus X^-,
\qquad
X^\pm=:P^\pm H^1(\R^N).
$$
We equip $H^1(\R^N)$ with the spectral norm
$$
\norm{u}^2
:=\norm{|L|^{1/2}u}_{L^2 (\R^N)}^2, \qquad u \in H^1 (\R^N).
$$
This norm is equivalent to the usual $H^1$ norm; see, e.g., \cite{Pankov2005}. If
$$
u=u^++u^-,\qquad u^\pm=P^\pm u,
$$
then
$$
\langle Lu,u\rangle
=\norm{u^+}^2-\norm{u^-}^2.
$$
We shall repeatedly use the continuous Sobolev embeddings
\begin{equation}\label{eq:sobolev-spectral}
\norm{u}_{L^r(\R^N)}\leq C_r\norm{u},
\qquad 2\leq r\le2^*.
\end{equation}
Since $L$ commutes with the lattice translations $\tau_k u:=u(\cdot+k)$, $k\in\Z^N$, the spaces $X^\pm$ and the spectral norm are invariant under $\tau_k$.

We now perform the large-parameter scaling. Put
\begin{equation}\label{eq:scaling}
u=\lambda^{-1/(q-2)}v.
\end{equation}
Then
$$
Lu=\lambda^{-1/(q-2)}Lv,
$$
while
$$
|u|^{p-2}u
=\lambda^{-(p-1)/(q-2)}|v|^{p-2}v
$$
and
$$
\lambda|u|^{q-2}u
=\lambda^{-1/(q-2)}|v|^{q-2}v.
$$
Hence \eqref{eq:original} is equivalent to
\begin{equation}\label{eq:scaled}
Lv
=\eps |v|^{p-2}v-|v|^{q-2}v,
\qquad
\eps:=\lambda^{-(p-2)/(q-2)}.
\end{equation}
More precisely, $u \in H^1 (\R^N)$ is a weak solution to \eqref{eq:original} if and only if $v \in H^1 (\R^N)$ given by \eqref{eq:scaling} is a weak solution to \eqref{eq:scaled}.
Moreover
$$
\lambda\to\infty
\quad\Longleftrightarrow\quad
\eps\to0^+.
$$
At least formally, the limiting equation is
\begin{equation}\label{eq:limit}
Lv=-|v|^{q-2}v.
\end{equation}
This is the pure defocusing case. Under \eqref{eq:gap}, existence of solutions to \eqref{eq:limit} is known, see \cite{Pankov2005}. 

\section{An \texorpdfstring{$\eps$}{epsilon}-dependent truncation}
\label{sect:3}

Fix once and for all a number
\begin{equation}\label{eq:delta-choice}
0<\delta<\delta_*:=
\min\left\{
1-\frac{2}{q},
\frac{q-2}{p-2},
\frac{q-1}{p-1}
\right\}.
\end{equation}
For $\eps>0$, define
$$
M_\eps:=\left(\frac{\delta}{\eps}\right)^{1/(p-q)}.
$$
Thus
\begin{equation}\label{eq:epsM}
\eps M_\eps^{p-q}=\delta,
\qquad
M_\eps\to\infty
\quad\text{as }\eps\to0^+.
\end{equation}

For $M>0$, define an odd continuous function $a_M:\R\to\R$ by
\begin{equation}\label{eq:aM}
a_M(t)=
\begin{cases}
|t|^{p-2}t,& |t|\leq M,\\
M^{p-q}|t|^{q-2}t,& |t|> M.
\end{cases}
\end{equation}
We set
$$
f_\eps(t):=|t|^{q-2}t-\eps a_{M_\eps}(t),
\qquad
F_\eps(t):=\int_0^t f_\eps(s)\,ds.
$$
The truncated auxiliary equation is
\begin{equation}\label{eq:truncated}
Lv=-f_\eps(v).
\end{equation}
On the set $|v|\leq M_\eps$, this is exactly \eqref{eq:scaled}.

The next lemma collects the quantitative properties of the truncation. The fact that every constant is independent of $\eps$ is crucial.

\begin{lemma}\label{lem:truncation}
For every $\eps>0$ and every $t\in\R$ the following hold:
\begin{align}
(1-\delta)|t|^q
&\leq f_\eps(t)t\leq |t|^q,\label{eq:ft-bounds}\\
\frac{1-\delta}{q}|t|^q
&\leq F_\eps(t)\le\frac1q|t|^q,\label{eq:F-bounds}\\
|f_\eps(t)|&\leq |t|^{q-1}.\label{eq:f-growth}
\end{align}
Moreover
\begin{enumerate}[label=\textup{(\roman*)}]
\item $f_\eps$ is odd, continuous, locally Lipschitz, and $f_\eps(t)t>0$ for $t\ne0$;
\item $f_\eps(t)=o(t)$ as $t\to0$;
\item $F_\eps(t)/t^2\to\infty$ as $|t|\to\infty$;
\item the map
$$
t\mapsto\frac{f_\eps(t)}{|t|}
$$
is strictly increasing on $(-\infty,0)$ and on $(0,\infty)$;
\item with
$$
\theta:=q(1-\delta)>2,
$$
one has the uniform Ambrosetti-Rabinowitz condition (cf. \cite{AmbrosettiRabinowitz1973})
\begin{equation}\label{eq:AR}
0<\theta F_\eps(t)\leq f_\eps(t)t,
\qquad t\ne0.
\end{equation}
\end{enumerate}
\end{lemma}

\begin{proof}
By oddness it suffices to consider $t\geq0$. If $0\leq t\leq M_\eps$, then
\begin{equation}\label{eq:f-inside}
f_\eps(t)=t^{q-1}-\eps t^{p-1}
=t^{q-1}\bigl(1-\eps t^{p-q}\bigr).
\end{equation}
Since $t\leq M_\eps$, \eqref{eq:epsM} gives
$$
0\le\eps t^{p-q}\le\eps M_\eps^{p-q}=\delta.
$$
Hence
$$
(1-\delta)t^{q-1}\leq f_\eps(t)\leq t^{q-1}.
$$
If $t\geq M_\eps$, then by \eqref{eq:aM} and \eqref{eq:epsM},
\begin{equation}\label{eq:f-outside}
f_\eps(t)
=t^{q-1}-\eps M_\eps^{p-q}t^{q-1}
=(1-\delta)t^{q-1}.
\end{equation}
This proves \eqref{eq:ft-bounds}, \eqref{eq:f-growth}, and the positivity statement.

We compute the primitive explicitly. For $0\leq t\leq M_\eps$,
\begin{equation}\label{eq:F-inside}
F_\eps(t)=\frac{t^q}{q}-\eps\frac{t^p}{p}.
\end{equation}
Thus $F_\eps(t)\leq t^q/q$, and, using $\eps t^{p-q}\le\delta$,
$$
F_\eps(t)
\geq t^q\left(\frac1q-\frac{\delta}{p}\right)
\geq \frac{1-\delta}{q}t^q,
$$
because
$$
\frac1q-\frac{\delta}{p}-\frac{1-\delta}{q}
=\delta\left(\frac1q-\frac1p\right)>0.
$$
For $t\geq M_\eps$, let $A_M(t)=\int_0^t a_M(s)\,ds$. Then
\begin{align*}
A_{M_\eps}(t)
&=\frac{M_\eps^p}{p}
+M_\eps^{p-q}\frac{t^q-M_\eps^q}{q}.
\end{align*}
Using $\eps M_\eps^{p-q}=\delta$ and $\eps M_\eps^p=\delta M_\eps^q$, we obtain
\begin{align}
F_\eps(t)
&=\frac{t^q}{q}-\eps A_{M_\eps}(t)\notag\\
&=\frac{1-\delta}{q}t^q
+\delta M_\eps^q\left(\frac1q-\frac1p\right).
\label{eq:F-outside}
\end{align}
The second term in \eqref{eq:F-outside} is nonnegative, proving the lower bound in \eqref{eq:F-bounds}. The upper bound follows either from \eqref{eq:F-outside} directly or by integrating $f_\eps(s)\leq s^{q-1}$.

The behavior at zero and infinity is now immediate from \eqref{eq:f-inside}, \eqref{eq:F-bounds}, and $q>2$.

It remains to prove the strict monotonicity of $f_\eps(t)/t$ on $(0,\infty)$. For $0<t<M_\eps$,
$$
\frac{f_\eps(t)}{t}=t^{q-2}-\eps t^{p-2},
$$
and therefore
\begin{align*}
\frac{d}{dt}\left(\frac{f_\eps(t)}{t}\right)
&=t^{q-3}\left((q-2)-\eps(p-2)t^{p-q}\right)\\
&\geq t^{q-3}\left((q-2)-(p-2)\delta\right)>0
\end{align*}
by \eqref{eq:delta-choice}. For $t>M_\eps$, \eqref{eq:f-outside} gives
$$
\frac{f_\eps(t)}{t}=(1-\delta)t^{q-2},
$$
which is strictly increasing. At $t=M_\eps$ the two expressions coincide
$$
M_\eps^{q-2}-\eps M_\eps^{p-2}
=(1-\delta)M_\eps^{q-2}.
$$
Oddness gives the corresponding monotonicity on $(-\infty,0)$.

Note that, for $0<t<M_\eps$,
$$
f_\eps'(t)
=(q-1)t^{q-2}-\eps(p-1)t^{p-2}
\geq\bigl((q-1)-(p-1)\delta\bigr)t^{q-2}>0,
$$
while for $t>M_\eps$,
$$
f_\eps'(t)=(q-1)(1-\delta)t^{q-2}.
$$
Thus $f_\eps$ is locally Lipschitz (in fact, piecewise $C^1$) and nondecreasing. Its a.e. derivative satisfies a bound of the form
$$
|f_\eps'(t)|\leq C|t|^{q-2},
$$
where $C$ depends on $p,q,\delta$, but not on $\eps$.

Finally, \eqref{eq:ft-bounds} and the upper estimate in \eqref{eq:F-bounds} yield
$$
f_\eps(t)t
\geq(1-\delta)|t|^q
=q(1-\delta)\frac{|t|^q}{q}
\geq q(1-\delta)F_\eps(t).
$$
Since $q(1-\delta)>2$ by \eqref{eq:delta-choice}, this is \eqref{eq:AR}.
\end{proof}

\section{The defocusing variational problem and the generalized Nehari-Pankov manifold}
\label{sect:4}

The natural energy functional for \eqref{eq:truncated} is $J_\eps : H^1 (\R^N) \rightarrow \R$,
$$
J_\eps(v)
=\frac12\left(\norm{v^+}^2-\norm{v^-}^2\right)
+\int_{\R^N}F_\eps(v)\,dx, \quad v \in H^1 (\R^N).
$$
It is more convenient to multiply it by $-1$ and define $\Phi_\eps : H^1 (\R^N) \rightarrow \R$ by
\begin{equation}\label{eq:Phi}
\Phi_\eps(v)
:=-J_\eps(v)
=\frac12\norm{v^-}^2-\frac12\norm{v^+}^2
-\int_{\R^N}F_\eps(v)\,dx.
\end{equation}
Both $J_\eps$ and $\Phi_\eps$ are of $C^1$ class and the critical points of $J_\eps$ and $\Phi_\eps$ coincide. In particular,
$$
\Phi_\eps'(v)(\varphi)
=-\langle Lv,\varphi\rangle
-\int_{\R^N}f_\eps(v)\varphi\,dx.
$$

For the functional \eqref{eq:Phi}, the positive quadratic space is $X^-$ and the negative quadratic space is $X^+$. Set
$$
Y:=X^-,\qquad Z:=X^+.
$$
Then $H^1=Y\oplus Z$, and, if $v=y+z$ with $y\in Y$, $z\in Z$,
$$
\Phi_\eps(y+z)
=\frac12\norm{y}^2-\frac12\norm{z}^2
-\int_{\R^N}F_\eps(y+z)\,dx.
$$
This is exactly the usual strongly indefinite focusing form, with the two spectral spaces interchanged.

Following \cite{Pankov2005,SzulkinWeth2009}, we define the generalized Nehari-Pankov manifold
$$
\cM_\eps
:=\left\{
v\in H^1(\R^N)\setminus Z \ : \
\Phi_\eps'(v)(v)=0,
\quad
\Phi_\eps'(v)(\zeta)=0\ \forall\zeta\in Z
\right\}.
$$
For $y\in Y\setminus\{0\}$, put
$$
\widehat E(y):=Z\oplus\R_+y
=\{ty+z:t\geq0,\ z\in Z\}.
$$
Let
$$
S_Y:=\{y\in Y:\norm{y}=1\}.
$$

We recall the following property, see \cite[Lemma 2.2]{SzulkinWeth2009}.

\begin{lemma}\label{lem:scalar}
Let $f:\R\to\R$ be odd, continuous, $f(t)t>0$ for $t\ne0$, and assume that $t\mapsto f(t)/|t|$ is strictly increasing on each of $(-\infty,0)$ and $(0,\infty)$. Let $F(t)=\int_0^t f(s)\,ds$. Then, for all $a,b\in\R$ and all $t\geq0$,
\begin{equation}\label{eq:scalar-ineq}
f(a)\left(\frac{t^2-1}{2}a+tb\right)
+F(a)-F(ta+b)
\le0.
\end{equation}
If $a\ne0$, equality in \eqref{eq:scalar-ineq} is possible only when $t=1$ and $b=0$.
\end{lemma}

We next recall, in the present notation, the existence result to the truncated problem via the Nehari-Pankov manifold approach. It follows directly from \cite[Theorem 1.1]{SzulkinWeth2009}, but we include the sketch of the proof for the sake of presentation. In particular, we give the details that will also be used for the uniform estimates in the next sections.

\begin{proposition}[{\cite[Theorem 1.1]{SzulkinWeth2009}}]\label{prop:ground-truncated}
For every $\eps>0$, the functional $\Phi_\eps$ possesses a nontrivial critical point $v_\eps \in H^1 (\R^N)$ such that
\begin{equation}\label{eq:c-eps}
c_\eps:=\Phi_\eps(v_\eps)
=\inf_{v\in\cM_\eps}\Phi_\eps(v)
=\inf_{y\in S_Y}\ \max_{w\in\widehat E(y)}\Phi_\eps(w)
>0.
\end{equation}
In particular, $v_\eps$ is a ground state of the truncated defocusing problem \eqref{eq:truncated}.
\end{proposition}

\begin{proof}

\noindent \textbf{Step 1: the geometry.}
Fix $y\in Y\setminus\{0\}$. Since $f_\eps(t)=o(t)$ at zero,
$$
\Phi_\eps(ty)
=\frac{t^2}{2}\norm{y}^2+o(t^2)>0
$$
for all sufficiently small $t>0$. On the other hand, $\Phi_\eps(z)\le0$ for $z\in Z$.

The functional $\Phi_\eps$ tends to $-\infty$ along unbounded sequences in the fiber $\widehat E(y)$. To see this, let
$$
w_n=t_ny+z_n\in\widehat E(y),
\qquad
\norm{w_n}\to\infty,
$$
and put $r_n=\norm{w_n}$ and $\widetilde w_n=w_n/r_n$. Passing to a subsequence,
$$
\frac{t_n}{r_n}\to t_0\geq0,
\qquad
\frac{z_n}{r_n}\weak z_0\in Z.
$$
If $t_0=0$, then, by orthogonality of $Y$ and $Z$,
$$
\frac{\norm{z_n}^2}{r_n^2}\to1.
$$
Since $F_\eps\geq0$,
$$
\frac{\Phi_\eps(w_n)}{r_n^2}
\leq \frac12\frac{t_n^2\norm y^2}{r_n^2}
-\frac12\frac{\norm{z_n}^2}{r_n^2}
\longrightarrow-\frac12.
$$
If $t_0>0$, then
$$
\widetilde w_n\weak t_0y+z_0\ne0,
$$
because $t_0y\in Y\setminus\{0\}$ and $z_0\in Z$. By local compactness, after passing to a subsequence $\widetilde w_n(x)\to\widetilde w(x)$ a.e., where $\widetilde w=t_0y+z_0\ne0$. On a set of positive measure, $\widetilde w(x)\ne0$; there
$$
|w_n(x)|=r_n|\widetilde w_n(x)|\to\infty.
$$
Since $F_\eps(t)/t^2\to\infty$, Fatou's lemma gives
$$
\frac1{r_n^2}\int_{\R^N}F_\eps(w_n)\,dx\to\infty.
$$
Hence $\Phi_\eps(w_n)\to-\infty$.

The function $F_\eps$ is convex because $f_\eps$ is nondecreasing (see the proof of Lemma \ref{lem:truncation}). Therefore $w\mapsto\int_{\R^N} F_\eps(w) \, dx$ is weakly lower semicontinuous. The negative quadratic term on $Z$ is weakly upper semicontinuous. It follows that $\Phi_\eps|_{\widehat E(y)}$ attains a positive global maximum at some point
$$
m_\eps(y)=t_yy+z_y,
\qquad t_y>0,
\quad z_y\in Z.
$$
At such a maximum the derivative vanishes in the $y$ direction and in all $Z$ directions, hence $m_\eps(y)\in\cM_\eps$.

Conversely, let $v\in\cM_\eps$. Write $v=y+z$ with $y\in Y\setminus\{0\}$ and $z\in Z$. For any $t\geq0$ and $\zeta\in Z$, the identities
$$
\Phi_\eps'(v)(v)=0,
\qquad
\Phi_\eps'(v)(\zeta)=0
$$
yield, after expanding the quadratic part,
\begin{align*}
\Phi_\eps(tv+\zeta)-\Phi_\eps(v)
&=-\frac12\norm{\zeta}^2\\
&\quad+\int_{\R^N}
f_\eps(v)\left(\frac{t^2-1}{2}v+t\zeta\right)
+F_\eps(v)-F_\eps(tv+\zeta)
 \, dx.
\end{align*}
Lemma \ref{lem:scalar}, applied pointwise with $a=v(x)$ and $b=\zeta(x)$, implies
$$
\Phi_\eps(tv+\zeta)\le\Phi_\eps(v).
$$
Strict monotonicity implies uniqueness of the maximizer on each fiber. Consequently the map
$$
S_Y\ni y\mapsto m_\eps(y)\in\cM_\eps
$$
is a well-defined bijection (in fact, it is an homeomorphism), and
$$
\inf_{v\in\cM_\eps}\Phi_\eps(v)
=\inf_{y\in S_Y}\max_{w\in\widehat E(y)}\Phi_\eps(w).
$$

Then, one can show (see \cite{SzulkinWeth2009}) that $\Psi$ defined as
$$
\Psi_\eps(y):=\Phi_\eps(m_\eps(y)),
\qquad y\in S_Y,
$$
is of class $C^1$, and for $\eta\in T_yS_Y$,
$$
\Psi_\eps'(y)(\eta)
=\norm{P_Ym_\eps(y)}\,
\Phi_\eps'(m_\eps(y))(\eta).
$$
Ekeland's variational principle applied to $\Psi_\eps$ therefore yields a sequence $(w_n)\subset\cM_\eps$ such that
\begin{equation}\label{eq:PSM}
\Phi_\eps(w_n)\to c_\eps,
\qquad
\Phi_\eps'(w_n)\to0.
\end{equation}

\noindent \textbf{Step 2: positivity and boundedness on $\cM_\eps$.}
Let $w=y+z\in\cM_\eps$, with $y\in Y=X^-$ and $z\in Z=X^+$. Since $\Phi_\eps'(w)(z)=0$ and $\Phi_\eps'(w)(w)=0$, we also have $\Phi_\eps'(w)(y)=0$. Thus
\begin{align*}
\norm{y}^2&=\int_{\R^N}f_\eps(w)y\,dx,\\
\norm{z}^2&=-\int_{\R^N}f_\eps(w)z\,dx.
\end{align*}
Using \eqref{eq:f-growth}, H\"older inequality, and \eqref{eq:sobolev-spectral},
\begin{align}
\norm{y}
&\leq C\norm{w}_{L^q(\R^N)}^{q-1},
\qquad
\norm{z}
\leq C\norm{w}_{L^q(\R^N)}^{q-1}.
\label{eq:component-bound-M}
\end{align}
Hence
\begin{equation}\label{eq:norm-vs-Lq-M}
\norm{w}\leq C\norm{w}_{L^q(\R^N)}^{q-1}
\leq C\norm{w}^{q-1}.
\end{equation}
Because $w\ne0$ and $q>2$, \eqref{eq:norm-vs-Lq-M} gives a number $\rho_0>0$, independent of $\eps$, such that
\begin{equation}\label{eq:M-away-zero}
\norm{w}\geq\rho_0,
\qquad \mbox{for all }  w\in\cM_\eps.
\end{equation}
Combining the first inequality in \eqref{eq:norm-vs-Lq-M} with \eqref{eq:M-away-zero}, we also obtain
\begin{equation}\label{eq:Lq-away-zero}
\norm{w}_{L^q(\R^N)}\geq\eta_0>0,
\qquad \mbox{for all } w\in\cM_\eps,
\end{equation}
with $\eta_0$ independent of $\eps$.

Moreover, since $\Phi_\eps'(w)(w)=0$,
\begin{align}
\Phi_\eps(w)
&=\Phi_\eps(w)-\frac12\Phi_\eps'(w)(w)=\int_{\R^N}
\frac12 f_\eps(w)w-F_\eps(w) \, dx.
\label{eq:energy-identity}
\end{align}
By \eqref{eq:ft-bounds} and \eqref{eq:F-bounds},
\begin{equation}\label{eq:kappa}
\frac12 f_\eps(t)t-F_\eps(t)
\geq\left(\frac{1-\delta}{2}-\frac1q\right)|t|^q
=:\kappa_\delta |t|^q, \qquad t \in \R,
\end{equation}
where $\kappa_\delta>0$ by \eqref{eq:delta-choice}. Hence, using \eqref{eq:Lq-away-zero},
\begin{equation}\label{eq:c-positive-uniform}
\Phi_\eps(w)\geq\kappa_\delta\norm{w}_{L^q(\R^N)}^q
\geq\kappa_\delta\eta_0^q>0.
\end{equation}
In particular $c_\eps>0$.

If $\Phi_\eps(w)$ is bounded on a subset of $\cM_\eps$, then \eqref{eq:c-positive-uniform} bounds $\norm w_{L^q}$; \eqref{eq:component-bound-M} then bounds $\norm w$. Thus $\Phi_\eps|_{\cM_\eps}$ is coercive. In particular, the Palais-Smale sequence in \eqref{eq:PSM} is bounded.

\noindent \textbf{Step 3: concentration-compactness.}
Let $(w_n)\subset\cM_\eps$ be the bounded Palais-Smale sequence from \eqref{eq:PSM}. Suppose it vanishes in the sense of Lions, i.e. for some $R>0$,
$$
\sup_{\xi \in\R^N}\int_{B(\xi, R)}|w_n|^2\,dx\to0.
$$
Then the Lions concentration-compactness lemma \cite{Lions1984} gives
\begin{equation}\label{eq:Lions-vanish}
\norm{w_n}_{L^q(\R^N)}\to0
\end{equation}
because $2<q<2^*$. But \eqref{eq:component-bound-M} would imply $\norm{w_n}\to0$, contradicting \eqref{eq:M-away-zero}. Therefore the sequence does not vanish.

Consequently, after replacing arbitrary concentration centers by nearby lattice points, there exist $R>0$, $\eta>0$, and $k_n\in\Z^N$ such that
\begin{equation}\label{eq:nonvanish-ball}
\int_{B(0,R)}|w_n(x+k_n)|^2\,dx\geq\eta.
\end{equation}
Set
$$
\widetilde w_n(x)=w_n(x+k_n).
$$
By periodicity, $(\widetilde w_n)$ is again a bounded Palais-Smale sequence for $\Phi_\eps$. After passing to a subsequence,
\begin{equation*}
\widetilde w_n\weak v_\eps
\quad\mbox{in }H^1(\R^N),
\end{equation*}
and
$$
\widetilde w_n\to v_\eps
\quad\text{in }L^r_{\mathrm{loc}}(\R^N),
\qquad 2\leq r<2^*,
$$
and a.e. Since $q<2^*$ and $|f_\eps(t)|\le|t|^{q-1}$, the Nemytskii operator is locally continuous from $L^q$ to $L^{q'}$. Hence, for every $\varphi\in C_c^\infty(\R^N)$,
$$
\Phi_\eps'(v_\eps)(\varphi)
=\lim_{n\to\infty}\Phi_\eps'(\widetilde w_n)(\varphi)=0.
$$
By density, $\Phi_\eps'(v_\eps)=0$. The strong convergence in $L^2_{\mathrm{loc}}$ together with \eqref{eq:nonvanish-ball} gives $v_\eps\ne0$.

By \eqref{eq:kappa}
$$
\frac12f_\eps(t)t-F_\eps(t) \geq 0, \qquad t \in \R.
$$
Since $v_\eps$ is a nontrivial critical point, $v_\eps\in\cM_\eps$, and therefore
$$
\Phi_\eps(v_\eps)
=\int_{\R^N} \frac12f_\eps(v_\eps)v_\eps-F_\eps(v_\eps)  \, dx.
$$
Fatou's lemma and \eqref{eq:energy-identity} yield
\begin{align*}
\Phi_\eps(v_\eps)
&\leq \liminf_{n\to\infty}
\int_{\R^N} \frac12f_\eps(\widetilde w_n)\widetilde w_n-F_\eps(\widetilde w_n)   \, dx\\
&=\liminf_{n\to\infty}\Phi_\eps(w_n)=c_\eps.
\end{align*}
On the other hand, $v_\eps\in\cM_\eps$ implies $\Phi_\eps(v_\eps)\geq c_\eps$. Thus equality holds and \eqref{eq:c-eps} follows.
\end{proof}

\section{Uniform comparison of the ground-state levels}
\label{sect:5}

The next step is the central uniform estimate. For $a>0$, consider the pure $q$-power functional $\Phi_a : H^1(\R^N) \rightarrow \R$,
\begin{equation*}
\Phi_a(v)
:=\frac12\norm{v^-}^2-\frac12\norm{v^+}^2
-\frac{a}{q}\norm{v}_{L^q(\R^N)}^q, \quad v = v^+ + v^- \in X^+ \oplus X^-.
\end{equation*}
Let $c_a$ be its Nehari-Pankov ground-state level. Proposition \ref{prop:ground-truncated}, with the obvious simplification, applies to $\Phi_a$ (\cite{SzulkinWeth2009}).

\begin{lemma}\label{lem:ca-scaling}
For every $a>0$,
\begin{equation}\label{eq:ca-scaling}
c_a=a^{-2/(q-2)}c_1.
\end{equation}
\end{lemma}

\begin{proof}
For $v\in H^1(\R^N)$, put $w=a^{1/(q-2)}v$, i.e.
$$
v=a^{-1/(q-2)}w.
$$
Then
\begin{align*}
\Phi_a(v)
&=\frac12a^{-2/(q-2)}
\left(\norm{w^-}^2-\norm{w^+}^2\right)
-\frac aq a^{-q/(q-2)}\norm w_q^q\\
&=a^{-2/(q-2)}\Phi_1(w),
\end{align*}
because
$$
1-\frac{q}{q-2}=-\frac{2}{q-2}.
$$
The scaling gives a bijection between the nontrivial critical points (and between the generalized Nehari manifolds) of $\Phi_1$ and $\Phi_a$. Taking the least critical level yields \eqref{eq:ca-scaling}.
\end{proof}

Let
$$
a_\delta:=1-\delta\in(0,1).
$$
By Lemma \ref{lem:truncation},
$$
\frac{a_\delta}{q}|t|^q
\leq F_\eps(t)
\le\frac1q|t|^q.
$$
Therefore, for $v \in H^1(\R^N)$,
\begin{equation}\label{eq:functional-order}
\Phi_1(v)
\le\Phi_\eps(v)
\le\Phi_{a_\delta}(v).
\end{equation}
The fiber minimax formula \eqref{eq:c-eps} preserves this ordering.

\begin{lemma}\label{lem:c-uniform}
For every $\eps>0$,
$$
c_1\leq c_\eps\leq c_{a_\delta}
=(1-\delta)^{-2/(q-2)}c_1.
$$
In particular, there is a constant $C_c>0$, independent of $\eps$, such that
$$
0<c_\eps\leq C_c.
$$
\end{lemma}

\begin{proof}
For every $y\in S_Y$, \eqref{eq:functional-order} gives
$$
\max_{w\in\widehat E(y)}\Phi_1(w)
\le
\max_{w\in\widehat E(y)}\Phi_\eps(w)
\le
\max_{w\in\widehat E(y)}\Phi_{a_\delta}(w).
$$
Taking the infimum over $y\in S_Y$ and using the fiber characterization \eqref{eq:c-eps} of the three ground-state levels gives
$$
c_1\leq c_\eps\leq c_{a_\delta}.
$$
Lemma \ref{lem:ca-scaling} finishes the proof.
\end{proof}

We can now bound the ground states $v_\eps$ uniformly.

\begin{proposition}\label{prop:H1-uniform}
There exist $c_H, C_H>0$, independent of $\eps>0$, such that the ground state $v_\eps$ in Proposition \ref{prop:ground-truncated} satisfies
\begin{equation}\label{eq:H1-uniform}
c_H \leq \norm{v_\eps}_{H^1(\R^N)}\leq C_H.
\end{equation}
\end{proposition}

\begin{proof}
Since $v_\eps$ is a critical point of $\Phi_\eps$,
$$
\Phi_\eps(v_\eps)
=\int_{\R^N}
\frac12 f_\eps(v_\eps)v_\eps-F_\eps(v_\eps) \, dx.
$$
By \eqref{eq:kappa} and Lemma \ref{lem:c-uniform},
\begin{equation*}
\kappa_\delta\norm{v_\eps}_{L^q(\R^N)}^q
\leq c_\eps\leq C_c.
\end{equation*}
Thus
\begin{equation}\label{eq:Lq-C}
\norm{v_\eps}_{L^q}\leq C_q
:=\left(\frac{C_c}{\kappa_\delta}\right)^{1/q}.
\end{equation}

Write
$$
v_\eps=v_\eps^-+v_\eps^+,
\qquad v_\eps^-\in X^-,\quad v_\eps^+\in X^+.
$$
Testing \eqref{eq:truncated} by $v_\eps^-$ gives
$$
-\norm{v_\eps^-}^2
=-\int_{\R^N}f_\eps(v_\eps)v_\eps^-\,dx,
$$
so
\begin{equation}\label{eq:minus-test}
\norm{v_\eps^-}^2
=\int_{\R^N}f_\eps(v_\eps)v_\eps^-\,dx.
\end{equation}
Similarly, testing by $v_\eps^+$ gives
\begin{equation}\label{eq:plus-test}
\norm{v_\eps^+}^2
=-\int_{\R^N}f_\eps(v_\eps)v_\eps^+\,dx.
\end{equation}
Using \eqref{eq:f-growth}, H\"older's inequality, and \eqref{eq:sobolev-spectral},
\begin{align*}
\norm{v_\eps^-}^2
&\le\norm{v_\eps}_{L^q(\R^N)}^{q-1}\norm{v_\eps^-}_{L^q(\R^N)}
\leq C_q^{q-1}C_S\norm{v_\eps^-},\\
\norm{v_\eps^+}^2
&\le\norm{v_\eps}_{L^q(\R^N)}^{q-1}\norm{v_\eps^+}_{L^q(\R^N)}
\leq C_q^{q-1}C_S\norm{v_\eps^+}.
\end{align*}
Hence
$$
\norm{v_\eps^-}+\norm{v_\eps^+}\leq C
$$
with $C$ independent of $\eps$. Norm equivalence gives the right inequality in \eqref{eq:H1-uniform}.

For the lower bound, the same two component estimates, without using \eqref{eq:Lq-C}, give
$$
\norm{v_\eps}
\leq C\norm{v_\eps}_{L^q(\R^N)}^{q-1}
\leq C\norm{v_\eps}^{q-1}.
$$
Since $v_\eps\ne0$ and $q>2$, division by $\norm{v_\eps}$ yields
$$
\norm{v_\eps}^{q-2}\geq C^{-1}.
$$
Thus $\norm{v_\eps}\geq c>0$, and norm equivalence gives the left part of \eqref{eq:H1-uniform}.
\end{proof}

\section{A uniform \texorpdfstring{$L^\infty$}{L-infinity} estimate}
\label{sect:6}

We now prove that the $H^1$ estimate in Proposition \ref{prop:H1-uniform}, together with the uniform growth bound \eqref{eq:f-growth}, implies a bound in $L^\infty$ independent of $\eps$. This is the step that removes the truncation for small $\eps$.

\begin{lemma}\label{lem:Linfty}
Let $(h_n)$ be a sequence of continuous functions satisfying
\begin{equation}\label{eq:h-growth}
|h_n(t)|\leq |t|^{q-1},
\qquad t\in\R,
\end{equation}
where $2<q<2^*$. Suppose $u_n\in H^1(\R^N)$ are weak solutions of
$$
-\Delta u_n+V(x)u_n=-h_n(u_n)
$$
and
$$
\sup_n\norm{u_n}_{H^1(\R^N)}\leq C_H.
$$
Then there exists $C_\infty>0$, depending only on $N,q,\norm{V}_{L^\infty(\R^N)},C_H$, such that
\begin{equation}\label{eq:Linfty-unif}
\sup_n\norm{u_n}_{L^\infty(\R^N)}\leq C_\infty.
\end{equation}
\end{lemma}

\begin{proof}
The argument is a finite local $W^{2,s}$ bootstrap, made uniform with respect to translations in $\R^N$.

By Sobolev embeddings,
$$
\sup_n\norm{u_n}_{L^{r_0}(\R^N)}\leq C,
\qquad r_0:=2^*=\frac{2N}{N-2}.
$$
Fix an arbitrary $x_0\in\R^N$. We shall use nested balls centered at $x_0$; the constants in the interior estimates do not depend on $x_0$ because the principal part is $-\Delta$ and $\norm{V}_{L^\infty(\R^N)}$ is fixed.

Assume at some stage that, on a ball $B(x_0,R)$,
$$
\norm{u_n}_{L^{r}(B(x_0,R))}\leq C_R
$$
for some $r\geq r_0$, uniformly in $n,x_0$. Put $s:=\frac{r}{q-1}$. Since $r\geq r_0=2^*>q-1$, we have $s>1$. By \eqref{eq:h-growth},
$$
\norm{h_n(u_n)}_{L^s(B(x_0,R))}
\le\norm{u_n}_{L^r(B(x_0,R))}^{q-1}\leq C.
$$
Also $s\leq r$, hence on the bounded ball
$$
\norm{V u_n}_{L^s(B(x_0,R))}
\leq C\norm{V}_{L^\infty(\R^N)}\norm{u_n}_{L^r(B(x_0,R))}.
$$
The Calder\'on-Zygmund inequality for
$$
-\Delta u_n=-V(x) u_n-h_n(u_n)
$$
gives, for some $0<r_1<R$,
\begin{equation}\label{eq:W2s}
\norm{u_n}_{W^{2,s}(B(x_0,r_1))}
\leq C\left(
\norm{Vu_n+h_n(u_n)}_{L^s(B(x_0,R))}
+\norm{u_n}_{L^s(B(x_0,R))}
\right)
\leq C.
\end{equation}
The constant is uniform in $n,x_0$.

If $2s>N$, the embedding $W^{2,s}\hookrightarrow L^\infty$ gives the desired local $L^\infty$ bound. Since $x_0$ is arbitrary, the global bound follows.

If $2s=N$, then $W^{2,N/2}$ embeds into $L^r$ for every finite $r$. We choose a sufficiently large finite exponent $r$ so that, at the next step, $r/(q-1)>N/2$; then the preceding case applies.

Suppose therefore that $2s<N$. Sobolev embedding for $W^{2,s}$ gives the new exponent
$$
r_{\mathrm{new}}
=\frac{Ns}{N-2s},
\qquad
\frac1{r_{\mathrm{new}}}
=\frac1s-\frac2N
=\frac{q-1}{r}-\frac2N.
$$
We iterate this construction on finitely many nested balls. To see that the iteration reaches the $L^\infty$ regime after finitely many steps, define $r_j$ by $r_0=2^*$ and, as long as $2r_j/(q-1)<N$,
$$
\frac1{r_{j+1}}
=(q-1)\frac1{r_j}-\frac2N.
$$
Let
$$
a_j:=\frac1{r_j},
\qquad
a_*:=\frac{2}{N(q-2)}.
$$
Then
$$
a_{j+1}-a_*
=(q-1)(a_j-a_*),
$$
so
\begin{equation}\label{eq:affine-solution}
a_j-a_*=(q-1)^j(a_0-a_*).
\end{equation}
Now
$$
a_0=\frac{N-2}{2N},
$$
and the subcritical condition $q<2^*=2N/(N-2)$ is exactly
$$
\frac{N-2}{2N}<\frac{2}{N(q-2)}=a_*.
$$
Hence $a_0-a_*<0$. Since $q-1>1$, the right-hand side of \eqref{eq:affine-solution} tends to $-\infty$. Therefore, after finitely many iterations, we must reach
$$
\frac{r_j}{q-1}>\frac N2,
$$
i.e. $2s_j>N$, or the borderline case $2s_j=N$, which was already handled.

Choose in advance a finite nested chain of balls, for example
$$
B(x_0, 2)\supset B(x_0, 2-1/m)\supset\cdots\supset B(x_0, 1),
$$
with $m$ larger than the finite number of bootstrap steps. Applying \eqref{eq:W2s} successively yields
$$
\norm{u_n}_{L^\infty(B(x_0,1))}\leq C_\infty
$$
with $C_\infty$ independent of $n$ and $x_0$. Taking the supremum over $x_0\in\R^N$ proves \eqref{eq:Linfty-unif}.
\end{proof}

Applying Lemma \ref{lem:Linfty} we easily obtain the following.

\begin{corollary}\label{cor:Linfty-veps}
There exists $C_\infty>0$, independent of $\eps>0$, such that
\begin{equation}\label{eq:veps-Linfty}
\norm{v_\eps}_{L^\infty(\R^N)}\leq C_\infty.
\end{equation}
\end{corollary}

\section{Removal of the truncation and proof of the main theorem}
\label{sect:7}

We are now in a position to remove the truncation.

\begin{proof}[Proof of Theorem \ref{thm:main}]
Choose
$$
\eps_0
:=\min\left\{1,\ \delta C_\infty^{-(p-q)}\right\}>0.
$$
If $0<\eps\le\eps_0$, then
$$
M_\eps
=\left(\frac\delta\eps\right)^{1/(p-q)}
\geq C_\infty.
$$
Together with \eqref{eq:veps-Linfty}, this gives
$$
|v_\eps(x)|\leq M_\eps
\quad\text{for a.e. }x\in\R^N.
$$
Therefore, by the first branch in \eqref{eq:aM},
$$
a_{M_\eps}(v_\eps)
=|v_\eps|^{p-2}v_\eps
\quad\text{a.e. in }\R^N.
$$
Equation \eqref{eq:truncated} hence becomes
\begin{align*}
Lv_\eps
&=-\left(|v_\eps|^{q-2}v_\eps
-\eps |v_\eps|^{p-2}v_\eps\right)=\eps |v_\eps|^{p-2}v_\eps-|v_\eps|^{q-2}v_\eps.
\end{align*}
Thus $v_\eps$ is a nontrivial solution of the \emph{untruncated} equation \eqref{eq:scaled} whenever $0<\eps\le\eps_0$.

Define
$$
\lambda_\infty
:=\eps_0^{-(q-2)/(p-2)}.
$$
For $\lambda\geq\lambda_\infty$, put
$$
\eps(\lambda)=\lambda^{-(p-2)/(q-2)}\le\eps_0
$$
and
\begin{equation}\label{eq:u-from-v}
u_\lambda
:=\lambda^{-1/(q-2)}v_{\eps(\lambda)}.
\end{equation}
By the calculation in Section 2, $u_\lambda$ solves \eqref{eq:original}. It is nontrivial because $v_{\eps(\lambda)}\ne0$.

Finally, Propositions \ref{prop:H1-uniform} and Corollary \ref{cor:Linfty-veps} yield
\begin{align*}
\norm{u_\lambda}_{H^1(\R^N)}
&=\lambda^{-1/(q-2)}\norm{v_{\eps(\lambda)}}_{H^1(\R^N)}
\leq C_H\lambda^{-1/(q-2)},\\
\norm{u_\lambda}_{L^\infty(\R^N)}
&=\lambda^{-1/(q-2)}\norm{v_{\eps(\lambda)}}_{L^\infty(\R^N)}
\leq C_\infty\lambda^{-1/(q-2)}.
\end{align*}
This proves Theorem \ref{thm:main}.
\end{proof}

\begin{remark}\label{rem:not-ground-original}
The solution $v_\eps$ is selected as a ground state of the truncated defocusing functional $\Phi_\eps$. Once the truncation is shown to be inactive on this solution, $v_\eps$ solves the original equation. We do not claim that it is a ground state among all solutions of the untruncated equation and the question about existence of a ground state remains open.
\end{remark}

\section{The limiting defocusing profile}
\label{sect:8}

The construction also explains the large-$\lambda$ asymptotics. Here we state the compactness conclusion.

\begin{proof}[Proof of Theorem \ref{thm:profile}]
For all sufficiently large $n$, recall that $v_n:=v_{\eps_n}$ satisfies the untruncated equation
\begin{equation}\label{eq:vn-exact}
Lv_n
=\eps_n|v_n|^{p-2}v_n-|v_n|^{q-2}v_n.
\end{equation}
Proposition \ref{prop:H1-uniform} gives
$$
\sup_n\norm{v_n}_{H^1(\R^N)}<\infty,
\qquad
\inf_n\norm{v_n}_{H^1(\R^N)}>0.
$$
We claim that $(v_n)$ cannot vanish in the sense of Lions. Otherwise
$$
\norm{v_n}_{L^q(\R^N)}\to0.
$$
Testing \eqref{eq:truncated} (or, equivalently, using \eqref{eq:minus-test}-\eqref{eq:plus-test}) and $|f_{\eps_n}(t)|\le|t|^{q-1}$ gives
$$
\norm{v_n}
\leq C\norm{v_n}_{L^q(\R^N)}^{q-1}\to0,
$$
contradicting the uniform lower bound. Hence, after lattice translations, there exist $R,\eta>0$ such that
$$
\int_{B(0,R)}|v_n(x+k_n)|^2\,dx\geq\eta.
$$
Set $\widetilde v_n=v_n(\cdot+k_n)$. Periodicity of $V$ preserves \eqref{eq:vn-exact}. Boundedness gives, after a subsequence,
$$
\widetilde v_n\weak v_0\quad\text{in }H^1(\R^N)
$$
and strong convergence in every $L^r_{\mathrm{loc}} (\R^N)$, $2 \leq r<2^*$. Moreover, $v_0\ne0$.

For the $p$-term,
\begin{equation}\label{eq:pterm-zero}
\norm{\eps_n|\widetilde v_n|^{p-2}\widetilde v_n}_{L^{p'}(\R^N)}
=\eps_n\norm{\widetilde v_n}_{L^p(\R^N)}^{p-1}
\leq C\eps_n\to0.
\end{equation}
For the $q$-term, strong convergence in $L^q_{\mathrm{loc}} (\R^N)$ implies
$$
|\widetilde v_n|^{q-2}\widetilde v_n
\to |v_0|^{q-2}v_0
\quad\text{in }L^{q'}_{\mathrm{loc}} (\R^N).
$$
Passing to the limit in the weak formulation of \eqref{eq:vn-exact} yields \eqref{eq:profile-eq}.
\end{proof}

\section*{Acknowledgements}

Bartosz Bieganowski was partly supported by the National Science Centre, Poland (Grant No. 2022/47/D/ST1/00487).

\section*{Statements and Declarations}

\textbf{Conflict of interest.} There is no conflict of interest.

\textbf{Data availability.} Not applicable.

\textbf{AI.} This manuscript used ChatGPT (OpenAI, GPT-5 Thinking) for light language editing only. No mathematical results, proofs or data were generated by the tool. All content was veriﬁed by the author.

\end{document}